\documentclass{amsart}
\usepackage{enumitem}
\usepackage{amsmath}
\usepackage{amsfonts}
\usepackage{mathtools}
\usepackage{amssymb}
\usepackage{amsthm}
\usepackage[all,cmtip]{xy}
\usepackage{tikz-cd}
\usepackage{xparse}
\newtheorem{theorem}{Theorem}
\newtheorem{lemm}{Lemma}
\newtheorem{defi}{Definition}

\newtheorem{coro}{Corollary}

\newtheorem{ex}{Example}
\newtheorem{rem}{Remark}

\begin{document}

\title{$n$-absorbing monomial ideals in polynomial rings}
\newcommand{\Ass}{\mbox{\textup{Ass}}} 
\newcommand{\NDC}{\mbox{\textup{NDC}}}
\newcommand{\Spec}{\mbox{\textup{Spec}}}
\newcommand{\Min}{\mbox{\textup{Min}}} 
\newcommand{\height}{\mbox{\textup{ht}}}
\newcommand{\supp}{\mbox{\textup{supp}}}
\newcommand{\reg}{\mbox{\textup{reg}}}
\newcommand{\adeg}{\mbox{\textup{adeg}}}

\author{Hyun Seung choi}
  \address{Department of Mathematics, Glendale Community College, Glendale, California 91208, U.S.A.}
\curraddr{}
\email{hyunc@glendale.edu}
\thanks{}

\author{Andrew Walker}
  \address{Department of Mathematics, College of Charleston, Charleston, South Carolina 29424, U.S.A.}
\curraddr{}
\email{walkeraj@cofc.edu}
\thanks{}



\date{December 13, 2018}

\dedicatory{}

\begin{abstract}
In a commutative ring $R$ with unity, given an ideal $I$ of $R$, Anderson and Badawi in 2011 introduced the invariant $\omega(I)$, which is the minimal integer $n$ for which $I$ is an $n$-absorbing ideal of $R$. In the specific case that $R = k[x_{1}, \ldots, x_{n}]$ is a polynomial ring over a field $k$ in $n$   variables $x_{1},\ldots, x_{n}$, we calculate $\omega(I)$ for certain monomial ideals $I$ of $R$.
\end{abstract}

\maketitle
\section{Introduction}

Throughout this paper, we set $\mathbb{N} := \{1,2,\ldots, \}$ and $\mathbb{N}_{0} := \{0,1,2, \ldots \}$ and $R$ will denote a commutative ring with unity. Given a nonzero ideal $I$ of $R$, Ass$(R/I)$ will denote the set of associated primes of $I$ in $R$. The primary notion we are interested in this paper is the following:

\begin{defi}
Let $n \in \mathbb{N}$, $R$ a commutative ring with unity, and $I$ an ideal of $R$. $I$ is said to be an \textbf{n-absorbing ideal} of a ring $R$ if for any $x_{1},\ldots, x_{n+1} \in R$ such that $x_{1} \cdots x_{n+1} \in I$, there is some $j$ with $1 \leq j \leq n+1$ such that $x_{1} \cdots x_{j-1}x_{j+1} \cdots x_{n+1} \in I$. $I$ is said to be a \textbf{strongly n-absorbing ideal} of a ring $R$ if for any  ideals $I_{1},\ldots, I_{n+1}$ of  $R$ such that $I_{1} \cdots I_{n+1} \subseteq I$, there is some $j$ with $1 \leq j \leq n+1$ such that $I_{1} \cdots I_{j-1}I_{j+1} \cdots I_{n+1} \subseteq I$.
\end{defi} 

(Strongly) $2$-absorbing ideals were initially defined and investigated by Badawi in \cite{Badawi} as a generalization of prime ideals, which are precisely the proper $1$-absorbing ideals. In 2011, Anderson and Badawi together generalized this further to the notion of a (strongly) $n$-absorbing ideal for any $n \in \mathbb{N}$ defined above in \cite{Anderson}. For an ideal $I$ in a ring $R$, we let $\omega(I)$ denote the minimal integer $n \in \mathbb{N}$ such that $I$ is $n$-absorbing. In a general ring, $I$ may not be $n$-absorbing for any $n \in \mathbb{N}$, in which case we set $\omega(I) = \infty$. Similarly, we can define the invariant $\omega^{\bullet}(I)$ to be the smallest integer $n \in \mathbb{N}$ for which an ideal $I$ is strongly $n$-absorbing, and set $\omega^{\bullet}(I) = \infty$ if no such integer exists. We set $\omega(R) = \omega^{\bullet}(R) = 0$. It is easy to see that 
$\omega(I)\le \omega^{\bullet}(I)$ holds for each ideal $I$ of $R$. In fact, Anderson and Badawi in Conjecture $1$ of \cite[page 1669]{Anderson} postulate that $\omega(I) = \omega^{\bullet}(I)$ holds for any ideal $I$ in arbitrary ring $R$; that is, they conjecture that the notion of an $n$-absorbing ideal and strongly $n$-absorbing ideal coincide. As of this writing, this problem remains open. However, it is known that the conjecture holds true for any $n\in\mathbb{N}$ if $R$ is a Pr{\"u}fer domain (\cite[Corollary 6.9]{Anderson}) or a commutative algebra over an infinite field(\cite{Donadze}), and for any ring $R$ if $n=2$ (\cite[Theorem 2.13]{Badawi}). The interested reader may refer to the survey article \cite[Section 5]{Badawi2} for  further information on strongly $n$-absorbing ideals. Recall that for an ideal $I$ in a ring $R$, the \textit{Noether exponent} of $I$, denoted by $e(I)$,  is the minimal integer $\mu \in \mathbb{N}$ such that $(\sqrt{I})^{\mu} \subseteq I$. If such an integer does not exist, we set $e(I) = \infty$. We also set $ e(R) = 0$. In a Noetherian ring, since $\sqrt{I}$ is finitely generated for any ideal $I$, $e(I) < \infty$. Anderson and Badawi in \cite{Anderson} establish a connection between $\omega^{\bullet}(I)$ and Noether exponents:

\begin{theorem}[Remark 2.2, Theorem 5.3, Section 6, Paragraph 2 on page 1669, \cite{Anderson}] 
\label{Wtheo}
Let $I_{1},\cdot\cdot\cdot, I_{r}$ be ideals of a ring $R$. Then $\omega(I_{1}\cap\cdot\cdot\cdot\cap I_{r})\le\omega(I_{1})+\cdot\cdot\cdot +\omega(I_{r})$ and  $\omega^{\bullet}(I_{1}\cap\cdot\cdot\cdot\cap I_{r})\le\omega^{\bullet}(I_{1})+\cdot\cdot\cdot +\omega^{\bullet}(I_{r})$. In particular, let $I$ be an ideal in a Noetherian ring $R$. If $I = Q_{1} \cap \cdots \cap Q_{n}$, where the $Q_{i}$ are primary ideals, then $\omega(I)\le\omega^{\bullet}(I) \leq \sum\limits\limits_{i=1}^{n} e(Q_{i})$. Thus every ideal in a Noetherian ring is $n$-absorbing for some $n \in \mathbb{N}$.
\end{theorem}

On the other hand, $e(I)$ is actually a lower bound of $\omega(I)$.

\begin{lemm}
\label{fund}
Given an ideal $I$ of a ring $R$, $e(I) \le \omega(I)$. If $Q$ is a primary ideal of $R$, then $\omega(Q) = \omega^{\bullet}(Q) = e(Q)$.
\end{lemm}

\begin{proof}
The first statement follows either  \cite[Corollary 3]{Walker} or \cite{Donadze2}. The second statements follows from 
\cite[Theorem 6.3(c), Theorem 6.6]{Anderson}.
\end{proof}

This raises the question then if for an arbitrary ideal $I$ whether $\omega(I)$ can be described purely in terms of Noether exponents or possibly other well-known ring-theoretic invariants. This has been investigated to some extent by others in at least one case. Namely, Moghimi and Naghani \cite[Theorem 2.21(1)]{Moghimi} show that in a discrete valuation ring $R$, $\omega(I)$ is precisely the length of the $R$-module $R/I$. 

In this spirit, we attempt to give in this paper a description of $\omega(I)$ in terms of other ring-theoretic invariants in the special case that $I$ is a monomial ideal of a polynomial ring over a field. In some cases, our arguments are general enough to also give the same results for $\omega^{\bullet}(I)$, and thus as a side-effect we can show that in some cases the notion of $n$-absorbing ideal and strongly $n$-absorbing ideal coincide as Anderson and Badawi conjecture.   

The present paper is divided into two parts. In section 2, we review the definitions and facts concerning $n$-absorbing ideals and monomial ideals. Using these we calculate $\omega(I)$ for primary monomial ideals by computing Noether exponents and the standard primary decomposition of monomial ideals. These results lead to the study of how $\omega(I)$ can be explicitly computed from the generating set of $I$ when $I$ is a monomial ideal of $R=k[x_{1},\cdot\cdot\cdot, x_{n}]$ with $n\le 3$ in the following section.

The second part is section 4, where we define and investigate $\omega$-linear monomial ideals, i.e., monomial ideals $I$ such that $\omega(I^{m})=m\omega(I)$ for each $m\in\mathbb{N}$. We give a characterization theorem for primary $\omega$-linear monomial ideals, and in particular show that every integrally closed monomial ideal in $R=k[x,y]$ and the edge ideal of a cycle is $\omega$-linear.

\section{Some Background}

As a prerequisite of the main section of this paper, we briefly review some of the basic material excerpted from \cite{Herzog} regarding monomial ideals, and show that $\omega(I)$ can be directly calculated from the generators of $I$ when $I$ is a primary monomial ideal.

Let $k$ be a field and $R=k[x_{1},\cdot\cdot\cdot, x_{n}]$ be the polynomial ring with $n$ variables over $k$. An element of $R$ of the form $x_{1}^{a_{1}}\cdot\cdot\cdot x_{n}^{a_{n}}$ with $a_{i}\in \mathbb{N}_{0}$ is called a \textit{monomial}, and an ideal of $R$ generated by monomials is called a \textit{monomial ideal}. The degree of $f=x_{1}^{a_{1}}\cdot\cdot\cdot x_{n}^{a_{n}}$, denoted by $\deg(f)$, is defined to be $a_{1}+\cdot\cdot\cdot+a_{n}$. $G(I)$ will denote the set of monomials in $I$ which are minimal with respect to divisibility. Any element of $R$ can be written uniquely as a $k$-linear combination of monomials; that is, given $f\in R$, we may write $f=\sum a_{u}u$ where the sum is taken over the monomial ideals of $R$ and $a_{u}\in k$ for each monomial $u$. Then the \textit{support} of $f$, denoted by $\text{supp}(f)$, is the set of monomials $u$ such that $a_{u}\neq 0$. An ideal $I$ of a ring $R$ is \textit{irreducible} if there are no ideals $I_{1}, I_{2}$ of $R$ such that $I=I_{1}\cap I_{2}$ and $I\subsetneq I_{1}$, $I\subsetneq I_{2}$. We denote by $\mathfrak{m}$ the unique maximal homogeneous ideal of $R$.


\begin{lemm}
\label{support}
(\cite[Chapter 1]{Herzog}) Let $R = k[x_{1},\ldots, x_{n}]$ and $I$ a monomial ideal of $R$ generated by monomials $u_{1},\cdot\cdot\cdot, u_{r}$ of $R$. Then the following hold:

\begin{enumerate}[label=(\roman*)]
\item Given a monomial $f\in I$, there exists $i\in\{1,\cdot\cdot\cdot, r\}$ so $u_{i}|f$.
\item $G(I)$ is the unique minimal set of monomial generators of $I$.
\item $I$ can be written as a finite intersection of ideals of the form $(x_{i_{1}}^{d_{1}},\ldots, x_{i_{m}}^{d_{m}})$. An irredundant presentation of this form is unique ($I=Q_{1}\cap\cdots\cap Q_{r}$ is \textrm{irredundant} if none of the ideals $Q_{i}$ can be omitted).

\item $I$ is irreducible if and only if $I$ is of the form $(x_{i_{1}}^{d_{1}},\ldots, x_{i_{m}}^{d_{m}})$. Moreover, every irreducible monomial ideal of the form $(x_{i_{1}}^{d_{1}},\ldots, x_{i_{m}}^{d_{m}})$ is $(x_{i_{1}},\ldots, x_{i_{m}})$-primary.

\item If $J$ is another monomial ideal of $R$, then
\[ I \cap J=(\{lcm(u,v)\mid u \in G(I), v \in G(J)\}).   \]
In particular, if $a$ and $b$ are coprime monomials of $R$ and $I$ is a monomial ideal of $R$, then $(ab,I)=(a,I)\cap(b,I)$.

\item An ideal $I'$ of $R$ is monomial if and only if for each $f\in I'$, $\supp(f)\subseteq I'$.

\end{enumerate}
\end{lemm}
By Lemma \ref{support}(iv), the irredundant unique decomposition of Lemma \ref{support}(iii) is also a primary decomposition of $I$, which is known as the \textit{standard decomposition} of $I$ (see \cite[P. 12]{Herzog}). We will also need the following characterization of primary monomial ideals:

\begin{lemm}
\label{pch}
\cite[Exercise 3.6]{Eisenbud}
Let $R=k[x_{1},\cdot\cdot\cdot, x_{n}]$ and  $P=(x_{i_{1}},\cdot\cdot\cdot, x_{i_{r}})$ a monomial prime ideal of $R$. Then given a $P$-primary monomial ideal $Q$, there exists $a_{1},\cdot\cdot\cdot, a_{r}\in \mathbb{N}$ and monomials $f_{1},\cdot\cdot\cdot, f_{s} \in k[x_{i_{1}},\cdot\cdot\cdot, x_{i_{r}}]$ so $G(Q)=\{x_{1}^{a_{1}},\cdot\cdot\cdot, x_{i_{r}}^{a_{r}}, f_{1},\cdot\cdot\cdot, f_{s}\}$. Conversely, every monomial ideal of this form is a $P$-primary ideal.
\end{lemm}


\begin{coro}
\label{monprimary}
Let $P$ be a prime monomial ideal and $I,J$ be $P$-primary monomial ideals of $R$. Then both $I\cap J$ and $IJ$ are $P$-primary monomial ideals. Moreover, $I:J$ is a $P$-primary monomial ideal provided $J \not\subset I$.
\end{coro}

\begin{proof}
This is an immediate consequence of Lemma \ref{support}.(v) and Lemma \ref{pch}. 
\end{proof}

We can now calculate $\omega(I)$, where $I$ is an irreducible monomial ideal. \\

\begin{lemm} \label{powermonomial}
Let $R = k[x_{1},\ldots,x_{n}]$ denote a polynomial ring over a field $k$. Let $I = (x^{d_{1}}_{i_{1}},\ldots, x_{i_{m}}^{d_{m}})$, where $d_{1},\ldots, d_{n} \in \mathbb{N}$ and $1 \leq i_{1} < i_{2} < \cdots < i_{m} \leq n$. Then $\omega(I)=\omega^{\bullet}(I)=e(I)=d_{1} + \cdots + d_{m} - m + 1$.
\end{lemm}

\begin{proof}
Since $I$ is a $(x_{i_{1}}, x_{i_{2}}, \cdot\cdot\cdot, x_{i_{m}})$-primary ideal by Lemma \ref{pch}, the first two equalities follow from Lemma \ref{fund}. Thus it suffices to show that $e(I)=r$, where $r=d_{1} + \cdots + d_{m} - m + 1$. We have $\sqrt{I} = (x_{i_{1}},\ldots, x_{i_{m}})$. For $N \in \mathbb{N}$, 
$(\sqrt{I})^{N} \subseteq I$ if and only if for every $c_{1},\ldots,c_{m} \in \mathbb{N}_{0}$ with $c_{1} + \cdots + c_{m} = N$, we have $x^{c_{1}}_{i_{1}} \cdots x^{c_{m}}_{i_{m}} \in I$. By Lemma \ref{support}(i), the latter happens precisely when $c_{i} \geq d_{i}$ for some $1 \leq i \leq m$.  Thus $e(I)=r$.  
\end{proof}

Next, we produce a way to calculate $\omega(I)$ when $I$ is a monomial primary ideal not necessarily generated by pure powers.

\begin{lemm}
\label{small}
Let $I$ be an ideal of a ring $R$. Suppose there is $P \in \Spec(R)$ such that $I=J_{1}\cap\cdot\cdot\cdot \cap J_{r}$, where $J_{i}$ are ideals of $R$ with $\sqrt{J_{i}}=P$ for each $i \in \{ 1, \ldots, r\}$. Then $e(I)=\max_{1\le i \le r} \{e(J_{i})\}$.
\end{lemm}  
  
\begin{proof}
Note that $\sqrt{I}=\sqrt{J_{1}}\cap\cdot\cdot\cdot\cap \sqrt{J_{r}}=P$. Thus given $\mu \in \mathbb{N}$, $(\sqrt{I})^{\mu} \subseteq I$ if and only if $(\sqrt{J_{i}})^{\mu} \subseteq J_{i}$ for each $i \in \{1, \ldots, r \}$, from which the conclusion of the lemma follows.
\end{proof}

\begin{coro}
\label{primary}
Let $R = k[x_{1},\ldots,x_{n}]$ denote a polynomial ring over a field $k$. If $Q$ is a monomial primary ideal of $R$  and $Q=\cap_{i=1}^{r}Q_{i}$ is its standard decomposition, then 
\begin{align*}
\omega(Q)=\omega^{\bullet}(Q)=\max_{1\le i\le r}\{e(Q_{i})\}.
\end{align*}


\end{coro}


\begin{ex}
Let $R=k[x,y,z]$ with a field $k$ and $I=(x^{4}, y^{3}, z^{2}, xy, y^{2}z)$. Then repeatedly applying Lemma \ref{support}(v), we obtain the standard decomposition $I=(x, y^{2}, z^{2}) \cap (x^{4}, y, z^{2}) \cap (x, y^{3}, z)$. Thus by Lemma \ref{powermonomial} and Corollary \ref{primary},
\begin{align*}
\omega(I)=\omega^{\bullet}(I)=\max\{1+2+2-3+1, 4+1+2-3+1, 1+3+1-3+1\}=5.
\end{align*}
\end{ex}

\section{When $I$ is a monomial ideal of $R=k[x_{1},\cdot\cdot\cdot, x_{n}]$ with $n\le 3$}
In this section we show that when $I$ is a monomial ideal of $R=k[x_{1},\cdot\cdot\cdot, x_{n}]$
 with $n\le 3$, then $\omega(I)$ can be explicitly calculated from $G(I)$. We first prove a theorem analogous to \cite[Theorem 2.5]{Anderson2}.

\begin{lemm} \label{irreducible}
Let $R$ be a UFD and $p$ an irreducible element of $R$. Then given $n\in\mathbb{N}$, $I$ is an $n$-absorbing ideal of $R$ if and only if $pI$ is an $(n+1)$-absorbing ideal of $R$. In particular, $\omega(pI) = \omega(I) + 1$. 
\end{lemm}

\begin{proof}
Suppose that $I$ is $n$-absorbing. Let $f_{1},\ldots, f_{n+2} \in R$ and $f_{1} \cdots f_{n+2} \in pI$. Then since $p$ is irreducible, $p \mid f_{i}$ for some $i$. Without loss of generality, suppose that $p \mid f_{1}$. Then $f_{1}/p \in R$, and so $(f_{1}/p)f_{2} \cdots f_{n+2} \in I$. Since $I$ is $n$-absorbing, and hence $(n+1)$-absorbing as well, we have that either $(f_{1}/p) f_{2} \cdots \widehat{f_{i}} \cdots f_{n+2} \in I$ for some $i \in \{2,\ldots, n+2\}$, in which case $f_{1}f_{2} \cdots \widehat{f_{i}} \cdots f_{n+2} \in pI$ and we're done, or $f_{2} \cdots f_{n+2} \in I$. This is a product of length $n+1$, so that since $I$ is $n$-absorbing, for some $j$ with $2 \leq j \neq n+1$, we have $f_{2} \cdots \widehat{f_{j}} \cdots f_{n+2} \in I$. Thus $p f_{2} \cdots \widehat{f_{j}} \cdots f_{n+2} \in pI$, and so $f_{1} f_{2} \cdots \widehat{f_{j}} \cdots f_{n+2} \in pI$. This shows that $pI$ is then $(n+1)$-absorbing, and 
$\omega(pI) \leq \omega(I)+1$.

To show the converse, suppose that $pI$ is an $(n+1)$-absorbing ideal. If $I$ is not an $n$-absorbing ideal, then there exists $f_{1},\cdot\cdot\cdot, f_{n+1} \in R$ such that $f=f_{1}\cdot\cdot\cdot f_{n+1} \in I$ but $f_{1}\cdot\cdot\cdot \widehat{f_{i}} \cdot\cdot\cdot f_{n+1} \not\in I$ for each $i$. Since $pI$ is $(n+1)$-absorbing and $pf \in pI$, it follows that either $pf_{1}\cdot\cdot\cdot \widehat{f_{i}} \cdot\cdot\cdot f_{n+1} \in pI$ for some $i$ or $f \in pI$. But the former is impossible by our choice of $f_{i}$'s, and without loss of generality we may assume that $p|f_{1}$. Now $(f_{1}/p)f_{2}\cdot\cdot\cdot f_{n} \in I$, and neither $\widehat{(f_{1}/p)}f_{2}\cdot\cdot\cdot f_{n+1}$ nor $(f_{1}/p)f_{2}\cdot\cdot\cdot \widehat{f_{i}} \cdot\cdot\cdot f_{n+1}$ is an element of $I$ for each $i\geq 2$.  
Therefore, since $R$ is a UFD, we may assume that none of $f_{i}$ are divisible by $p$. Now $pf_{1}\cdot\cdot\cdot f_{n+1} \in pI$, but $pf_{1}\cdot\cdot\cdot \widehat{f_{i}} \cdot\cdot\cdot f_{n+1} \not\in pI$ and $f_{1}\cdot\cdot\cdot f_{n+1} \not\in pI$, which contradicts the assumption that $pI$ is an $(n+1)$-absorbing ideal. Hence $I$ is an $n$-absorbing ideal and $\omega(pI)\ge \omega(I)+1$. 
\end{proof}

The following corollary is now immediate.

\begin{coro}
\label{prinmo}
Given a monomial $f$ and an ideal $I$ of $R=k[x_{1},\cdot\cdot\cdot, x_{n}]$, $\omega(fI)=$deg$(f)+\omega(I)$. In particular, $\omega(fR)=$deg$(f)$.
\end{coro}

Given a monomial ideal $I$ with the standard decomposition $I=\cap_{\ell=1}^{t}T_{\ell}$, we can define an equivalence relation on $\{1,\cdot\cdot\cdot, t\}$ by defining $i\sim j$ iff $\sqrt{T_{i}}=\sqrt{T_{j}}$, and set $\{S_{i}\}_{i=1}^{r}$ to be the corresponding equivalence classes. Then $Q_{i}=\cap_{\ell\in S_{i}}T_{\ell}$ is a monomial primary ideal for each $i \in \{1, \ldots, r\}$, and $I=\cap_{i=1}^{r}Q_{i}$ is an irredundant primary decomposition of $I$.  We will call this decomposition the \textit{canonical primary decomposition} of $I$.

\begin{lemm}
\label{l}
Let $R=k[x_{1},\cdot\cdot\cdot, x_{n}]$. Let $I$ be a monomial ideal with canonical primary decomposition $I=\cap_{i=1}^{r}{Q_{i}}$. If there exists $k \in\{1,\cdot\cdot\cdot, r\}$ such that $\sqrt{Q_{i}}\subseteq \sqrt{Q_{k}}$ for all $i\in\{1,\cdot\cdot\cdot, r\}$, then $\omega(I)=\max\{e(Q_{k}), \omega(\cap_{1\le i\le r, i\neq k}Q_{i})\}$ and $\omega^{\bullet}(I)=\max\{e(Q_{k}), \omega^{\bullet}(\cap_{1\le i\le r, i\neq k}Q_{i})\}$.
\end{lemm}

\begin{proof}
Let $t=\max\{e(Q_{k}), \omega(\cap_{1\le i\le r, i\neq k}Q_{i})\}$. We will first  show that $I$ is $t$-absorbing. If not, then there are $f_{1},\ldots, f_{t+1} \in R$ such that $f=\prod_{j=1}^{t+1}f_{j} \in I$ but $g_{j} :=f/f_{j}\not\in I$ for each $j \in\{1,\ldots, t+1\}$. Hence given any $i \in\{1,\ldots, t+1\}$, there exists $\ell \in \{1,\ldots, r\}$ such that $g_{i} \not\in Q_{\ell}$, and since $f_{i}g_{i}=f \in I\subseteq Q_{\ell}$,  we must have $f_{i}\in \sqrt{Q_{\ell}}\subseteq \sqrt{Q_{k}}$. 
Therefore, $g_{j}\in (\sqrt{Q_{k}})^{t}\subseteq (\sqrt{Q_{k}})^{e(Q_{k})}\subseteq Q_{k}$ for all $j \in\{1,\ldots, t+1\}$. On the other hand, $\cap_{1\le i\le r, i\neq k}Q_{i}$ is $t$-absorbing and $f\in \cap_{1\le i\le r, i\neq k}Q_{i}$, so that we conclude $g_{j} \in  \cap_{1\le i\le r, i\neq k}Q_{i}$ for some $j \in\{1,\ldots, t+1\}$ and thereby $g_{j}\in I$, a contradiction. Thus $\omega(I) \leq t$. Next, we show that $\omega(I)\ge t$; that is, $I$ is not $(t-1)$-absorbing. 
We now consider two cases.\\
\\
Case $1$:  $t = \omega( \cap_{1 \leq i \leq r, i \neq k} Q_{i})$. Since $\cap_{1\le i\le r, i\neq k}Q_{i}$ is not $(t-1)$-absorbing, there are $h_{1},\ldots, h_{t} \in R$ such that $h=\prod_{i=1}^{t}h_{i} \in \cap_{1\le i\le r, i\neq k}Q_{i}$ and $\ell_{j} := h/h_{j} \not\in \cap_{1\le i\le r, i\neq k}Q_{i}$ for each $j\in\{1,\ldots, t\}$. By an argument similar to the first paragraph of this proof, $h_{i}\in \sqrt{Q_{k}}$ for each $i\in \{1,\ldots, t\}$, and so $h \in Q_{k}$. Hence $h \in I$ and $\ell_{j}\not\in I$ for each $j\in \{1,\ldots, t\}$, so that $I$ is not $(t-1)$-absorbing.\\
\\
Case $2$: $t = e(Q_{k})$. Consider the standard decomposition of $I$, and choose an irreducible component $T$ of $I$ such that $e(T)=e(Q_{k})$ and $\sqrt{T}=\sqrt{Q_{k}}$. Since we obtained the canonical primary decomposition $I=\cap_{i=1}^{r} Q_{i}$ from the standard decomposition, we can choose a monomial $g\in (\cap_{1\le i\le r, i\neq k}Q_{i})\setminus T$ by Lemma \ref{support}(vi). Now $T=(x_{i_{1}}^{a_{1}},\cdot\cdot\cdot, x_{i_{l}}^{a_{l}})$ for some $a_{j} \in \mathbb{N}$ and $1\le i_{1}<\cdot\cdot\cdot<i_{l}\le n$. Note that we may assume that $g=\prod_{j=1}^{l}x_{i_{j}}^{c_{j}}$ for some $c_{j}\in\mathbb{N}_{0}$ such that $c_{j}<a_{j}$ for each $j\in\{1,\cdot\cdot\cdot, l\}$. Set \[f :=x_{i_{1}}^{a_{1}-1}\cdot\cdot\cdot x_{i_{l}}^{a_{l}-1}(x_{i_{1}}+\cdot\cdot\cdot+x_{i_{l}}).\] Then $f$ is a product of $e(T)$ elements of $\sqrt{T}$ by Lemma \ref{powermonomial}, and so $f \in (\sqrt{T})^{e(T)}= (\sqrt{Q_{k}})^{e(Q_{k})}\subseteq Q_{k}$. Since $g\mid f$ it also follows that $f \in \cap_{1\le i\le r, i\neq k}Q_{i}$. Hence $f \in I$.
However, given $j \in\{1,\cdot\cdot\cdot, l\}$, $\cfrac{f}{x_{i_{j}}}\not\in T$. Indeed, $x_{i_{1}}^{a_{1}-1}\cdot\cdot\cdot x_{i_{l}}^{a_{l}-1}\in \supp(\cfrac{f}{x_{i_{j}}})\setminus T$ by Lemma \ref{support}.(i), and $\cfrac{f}{x_{i_{j}}}\not\in T$ by Lemma \ref{support}.(vi). Similarly $x_{i_{1}}^{a_{1}-1}\cdot\cdot\cdot x_{i_{l}}^{a_{l}-1}=\cfrac{f}{x_{i_{1}}+\cdot\cdot\cdot+x_{i_{l}}}\not\in T$. Therefore $I$ is not $(e(Q_{k})-1)$-absorbing, and $\omega(I)\ge e(Q_{k})=t$. Hence we have shown that $\omega(I)=\max\{e(Q_{k}), \omega(\cap_{1\le i\le r, i\neq k}Q_{i})\}$. The proof of $\omega^{\bullet}(I)=\max\{e(Q_{k}), \omega^{\bullet}(\cap_{1\le i\le r, i\neq k}Q_{i})\}$ can be obtained in a similar manner, and is omitted.
\end{proof}

 The following corollary is immediate.
 
\begin{coro}
\label{refined}
Let $R=k[x_{1},\cdot\cdot\cdot, x_{n}]$ and $I$ be a monomial ideal of $R$ with standard decomposition $I=\bigcap_{i=1}^{r}T_{i}$. Then $\omega(I)=\omega^{\bullet}(I)=\max_{1\le i\le r}\{e(T_{i})\}$ if Ass$(R/I)$ is totally ordered under set inclusion. 
\end{coro}

In the next lemma, 
we give a characterization of when the upper bound of $\omega(I)$ from Theorem \ref{Wtheo} is sharp.
\begin{lemm}
\label{upper}
Let $I$ be a monomial ideal of $R=k[x_{1},\cdot\cdot\cdot, x_{n}]$ with an irredundant primary decomposition $I = Q_{1} \cap \cdots \cap Q_{r}$. 
Then $\omega(I)=\omega^{\bullet}(I)=\sum\limits_{i=1}^{r}e(Q_{i})$ if and only if $I$ has no embedded associated primes.
\end{lemm}

\begin{proof}
Set $P_{i}=\sqrt{Q_{i}}$ for each $i=1,\cdot\cdot\cdot, r$.\\
\\
$\Leftarrow$: Assume that $P_{1},\cdot\cdot\cdot, P_{r}$ are incomparable prime ideals. The case when $r=1$ follows from Lemma \ref{fund}, so we may assume that $r\ge 2$. Since $\omega(I)\le\omega^{\bullet}(I)\le\sum\limits_{i=1}^{r}e(Q_{i})$ by Theorem \ref{Wtheo}, it suffices to show that $I$ is not $( \sum\limits_{i=1}^{r} e(Q_{i})-1)$-absorbing. Now given $i \in \{1,\cdot\cdot\cdot, r\}$, choose $T_{i}$ to be an irreducible component of $I$ with $\sqrt{T_{i}} = P_{i}$ and $e(T_{i}) = e(Q_{i})$.  Write $T_{i} = (x_{i_{1}}^{a_{1}},\cdot\cdot\cdot, x_{i_{s_{i}}}^{a_{s_{i}}})$
 with $1\le i_{1}<\cdot\cdot\cdot< i_{s_{i}}\le n$ and $a_{1},\cdot\cdot\cdot, a_{s_{i}}\in \mathbb{N}$. For $i \in \{1, \ldots, r\}$ and $j \in \{1, \ldots, s_{i}\}$, set \[f_{i,j}=x_{i_{j}}+\sum\limits_{t\neq j}x_{i_{t}}^{2} \text{ and } f_{i}= \displaystyle \Big(\sum\limits_{l=1}^{s_{i}}x_{i_{l}}\Big)\Big(\prod_{j=1}^{s_{i}}f_{i,j}^{a_{j}-1}\Big).\] It follows that $f_{i}\in P_{i}^{e(T_{i})}=(\sqrt{Q_{i}})^{e(Q_{i})}\subseteq Q_{i}$.
Thus $f :=\prod_{i=1}^{r}f_{i} \in I$, and $f$ is a product of $\sum\limits_{i=1}^{r}e(Q_{i})$ elements of $R$. We wish to show that $\cfrac{f}{\sum\limits_{l=1}^{s_{i}}x_{i_{l}}}\not\in I$ and $\cfrac{f}{f_{i,j}}\not\in I$ for each $i\in\{1,\cdot\cdot\cdot, r\} \text{ and } j\in\{1,\cdot\cdot\cdot, s_{i}\}$. 
 Without loss of generality, we let $i = 1$. Note that $\cfrac{f_{1}}{f_{1,j}}\not\in T_{1}$,  since $\prod_{t=1}^{s_{1}}x_{1_{t}}^{a_{t}-1}\in \supp\Big(\cfrac{f_{1}}{f_{1,j}}\Big)\setminus T_{1}$. On the other hand, $\sum\limits_{l=1}^{s_{i}}x_{i_{l}}\not\in P_{1}$ and $f_{i,l}\not\in P_{1}$ for each $i\neq 1$ and $l\in\{1,\cdot\cdot\cdot, s_{i}\}$. Therefore $f_{i}\not\in P_{1}$ for each $i\neq 1$, and $\cfrac{f}{f_{1}}=\prod_{i=2}^{r}f_{i}\not\in P_{1}$. Hence $\cfrac{f}{f_{1,j}}=(\cfrac{f}{f_{1}})(\cfrac{f_{1}}{f_{1,j}})\not\in Q_{1}$. The proof that $\cfrac{f}{\sum\limits_{l=1}^{s_{1}}x_{1_{l}}}\not\in Q_{1}$ follows similarly. Hence we have $\omega(I)=\sum\limits_{i=1}^{r}e(Q_{i})$. \\
\\
$\Rightarrow$: We prove the contrapositive; assume that $P_{1},\cdot\cdot\cdot, P_{r}$ are not incomparable prime ideals. Then without loss of generality we may assume that $P_{1}\subsetneq P_{2}$, and we have $\omega(Q_{1}\cap Q_{2})=\max\{e(Q_{1}), e(Q_{2})\}$ by Corollary \ref{refined}. Therefore by Theorem \ref{Wtheo} we have
\begin{align*}
\omega(I)&=\omega \Big(Q_{1}\cap Q_{2}\cap(\bigcap_{i\neq 1,2} Q_{i})\Big)\\
&\le \omega \Big(Q_{1}\cap Q_{2} \Big)+\omega \Big(\bigcap_{i\neq 1,2} Q_{i} \Big)\\
&=\max\{e(Q_{1}), e(Q_{2})\}+\omega \Big(\bigcap_{i\neq 1,2} Q_{i}\Big)\\
&\le\max\{e(Q_{1}), e(Q_{2})\}+\sum_{i\neq 1,2} e(Q_{i})\\
&<\sum_{i=1}^{r}e(Q_{i}).
\end{align*}
\end{proof}
Lemma \ref{l} and Lemma \ref{upper} yield the following corollary.

\begin{coro}
\label{n-1}
Let $I$ is a monomial ideal of $R=k[x_{1},\cdot\cdot\cdot, x_{n}]$ with  dim$(R/I)=1$. Let $I=\cap_{i=1}^{r} Q_{i}$ be the canonical primary decomposition of $I$. Then 
 \[ \omega(I) = \omega^{\bullet}(I)=\left\{
\begin{array}{ll}
      \max\{e(Q_{k}),\sum\limits_{i\neq k}e(Q_{i})\}  & \text{ if } \sqrt{Q_{k}}= \mathfrak{m} \text{ for some } k\in\{1,\cdot\cdot\cdot, r\}. \\
       \sum\limits_{i=1}^{r}e(Q_{i})  & \text{ otherwise }.\\
   \end{array} 
\right. \]
\end{coro}

\begin{coro}
\label{realprinmo}
Let $f$ be a monomial of $R$. Then $\omega^{\bullet}(fR)=deg(f)$. In particular, $\omega(fR)=\omega^{\bullet}(fR)$.
\end{coro}

\begin{proof}
Let $f=\prod_{k=1}^{r}x_{i_{k}}^{a_{k}}$ for some $a_{1},\ldots, a_{r} \in \mathbb{N}$ and $1 \leq i_{1} < i_{2} < \cdots < i_{r} \leq n$. Then $fR=x_{i_{1}}^{a_{1}}R\cap\cdot\cdot\cdot \cap x_{i_{r}}^{a_{r}}R$, and by Lemma \ref{powermonomial} and Lemma \ref{upper} we have $\omega(fR)=\omega^{\bullet}(fR)=\sum_{i=1}^{r}e(x_{i_{k}}^{a_{k}}R)=\sum_{i=1}^{r} a_{k}=$deg$(f)$. 
\end{proof}



Given a monomial ideal $I$ of $R=k[x,y,z]$, we can produce an algorithm that can compute $\omega(I)$. If $I$ is principal, then Corollary \ref{realprinmo} says that $\omega(I )$ is equal to the degree of a generator for $I$. Otherwise, $I=hJ$ for some monomial $h$ and a monomial ideal $J$ with dim$(R/J)\le1$. 
Now, $\omega(J)$ can be calculated explicitly using Corollary \ref{primary} or Corollary \ref{n-1} after obtaining a canonical primary decomposition of $I$, and we have $\omega(I)=deg(h)+\omega(J)$ by Corollary \ref{prinmo}. 

\begin{ex}
Let $R=k[x,y,z]$ and $I=(x^{3}y^{4}, x^{2}y^{5}, x^{4}y^{3}z^{2}, x^{5}y^{3}z, x^{2}y^{4}z^{2})$. Then
$I=x^{2}y^{3}J$ with canonical primary decomposition $J=(x^{2},y)\cap (y,z) \cap (x^3, y^2, z^2, xy)$. By Lemma \ref{powermonomial} and Corollary \ref{primary}, the standard decomposition $(x^3, y^2, z^2, xy)=(x, y^{2}, z^{2})\cap (x^{3}, y, z^{2})$ yields that $e((x^3, y^2, z^2, xy))=4$. Thus by Corollary \ref{n-1},
 \begin{align*}
 \omega(I)&=deg(x^{2}y^{3})+\omega(J)\\
 &=5+\max\{e((x^3, y^2, z^2, xy)), e((x^{2},y))+e((y,z))\}\\
 &=5+\max\{4, 2+1\}\\
 &=9.
 \end{align*}  
\end{ex}

Another interesting corollary of Lemma \ref{irreducible} and Lemma  \ref{l} is a formula of $\omega(I)$ and $\omega^{\bullet}(I)$ for monomial ideals of $R=k[x,y]$ where $k$ is a field and $x,y$ are  indeterminates over $k$. 

\begin{theorem}
\label{hoo}
Let $R = k[x,y]$ and $J$ a monomial ideal of $R$.  Write $J=(x^{a_{1}}y^{b_{1}},\cdot\cdot\cdot, x^{a_{r}}y^{b_{r}})$, where $\{a_{i} \}$ is strictly decreasing and $\{ b_{i} \}$ is strictly increasing. Then
 \[ \omega(J) =\omega^{\bullet}(J)= \left\{
\begin{array}{ll}
      a_{1}+b_{1}  & \text{ if } r=1. \\
      \max_{1\le i \le r-1}\{a_{i}+b_{i+1}\}-1  & \text{ if } r>1.\\
   \end{array} 
\right. \]
\end{theorem}

\begin{proof}
The case when $r=1$ follows from Corollary \ref{realprinmo}. For $r > 1$, first observe the standard decomposition of $J$ is $J= x^{a_{r}}R\cap y^{b_{1}}R \cap (x^{a_{1}}, y^{b_{2}})\cap (x^{a_{2}}, y^{b_{3}}) \cap \cdot\cdot\cdot \cap (x^{a_{r-1}},y^{b_{r}})$ (\cite[Proposition 3.2]{Miller}). 
The case $b_{1}=a_{r}=0$ follows from Corollary \ref{primary}. Suppose that at least one of $a_{r}$ and $b_{1}$ is nonzero. Thus by Lemma \ref{powermonomial} and Corollary \ref{n-1},
\begin{align*}
\omega(J)=\omega^{\bullet}(J)&=\max\{e((x^{a_{1}}, y^{b_{2}})\cap (x^{a_{2}}, y^{b_{3}}) \cap \cdot\cdot\cdot \cap (x^{a_{r-1}},y^{b_{r}})), e(x^{a_{r}}R)+e(y^{b_{1}}R)\}\\
&=\max\{\max_{1\le i\le r-1}\{e((x^{a_{i}}, y^{b_{i+1}}))\}, a_{r}+b_{1}\}\\
&=\max\{\max_{1\le i\le r-1}\{a_{i}+b_{i+1}-1\}, a_{r}+b_{1}\}\\
&=\max_{1\le i\le r-1}\{a_{i}+b_{i+1}\}-1.
\end{align*}
\end{proof}

\begin{ex}
If $R=k[x,y]$ and $J=(x^{11}y^{4}, x^{8}y^{5}, x^{7}y^{9}, x^{4}y^{10}, x^{2}y^{16})$, then by Theorem \ref{hoo},\\
 $\omega(J)=\omega^{\bullet}(J)=\max\{11+5, 8+9, 7+10, 4+16\}-1=19$.
\end{ex}

\section{$\omega$-linear ideals}

Given an ideal $I$ of a ring $R$, we will say that $I$ is an \textit{$\omega$-linear ideal} if $\omega(I^{m})=m\omega(I)$ for each $m\in\mathbb{N}$. Perhaps the most common example of $\omega$-linear ideals can be found amongst those $P \in \Spec(R)$ which $P^{n}$ is $P$-primary for each $n \in \mathbb{N}$ (\cite[Theorem 3.1, Theorem 5.7]{Anderson}). For instance, \\
\\
1. $R$ is a Pr{\"u}fer domain and $P^{2}\neq P$.\\
2. $R$ is a Noetherian ring and $P$ is a maximal ideal that contains a nonzerodivisor.\\
3. $R=k[x_{1},\cdot\cdot\cdot, x_{n}]$ and $P$ is a monomial ideal.\\
\\
In this section, we investigate the properties of $\omega$-linear ideals. Again, we will restrict our concern to monomial ideals of a polynomial ring $R=k[x_{1},\cdot\cdot\cdot, x_{n}]$ where $k$ is a field. 

We first consider a few useful inequalities regarding monomial ideals.


\begin{lemm}
\label{topi}
Let $I$ be a monomial ideal of $R = k[x_{1},\ldots,x_{n}]$. Then $\omega(I)\ge \max\{\deg(f)\mid f\in G(I)\}$.
\end{lemm}

\begin{proof}
Let $f\in G(I)$. Then $f=\prod_{k=1}^{r}x_{i_{k}}^{a_{k}}$ for some $a_{1},\ldots, a_{r} \in \mathbb{N}$ and $1 \leq i_{1} < i_{2} < \cdots < i_{r} \leq n$. Since $f\in I$ but $\cfrac{f}{x_{i_{k}}}\not\in I$ for each $k\in\{1,\cdot\cdot\cdot, r\}$ by minimality of $G(I)$, we have that $I$ is not $($deg$(f)-1)$-absorbing. Hence $\omega(I)\ge$ deg$(f)$, and since $f$ was chosen arbitrarily, we have the desired conclusion.
\end{proof}

\begin{lemm}
\label{orderreversing}
Let $I\subseteq J$ be ideals of a ring $R$. If $\sqrt{I}=\sqrt{J}$, then $e(J)\le \omega(I)$. In particular, if $I$ and $J$ are both $P$-primary ideals of a prime ideal $P$ of $R$, then $\omega(J)\le \omega(I)$.
\end{lemm}

\begin{proof}
Since $\sqrt{I}=\sqrt{J}$, $(\sqrt{J})^{\omega(I)}\subseteq(\sqrt{I})^{e(I)}\subseteq I\subseteq J$ by Lemma \ref{fund} and $e(J)\le \omega(I)$. The second statement follows from that $e = \omega$ for primary ideals.
\end{proof} 

\begin{lemm}
\label{idealop}
Let $P$ be a prime monomial ideal of $R = k[x_{1},\ldots,x_{n}]$. If $I,J$ are $P$-primary monomial ideals of $R$, then $\omega(I+J)\le\min\{\omega(I),\omega(J)\} \le \max\{\omega(I),\omega(J)\}=\omega(I\cap J)\le\omega(IJ)\le \omega(I)+\omega(J)$. Moreover, $\omega(I:J)\ge \omega(I) - \omega(J)$.
\end{lemm}

\begin{proof}
Note that by Corollary \ref{monprimary}, $IJ\subseteq I\cap J\subseteq I+J$ are all $P$-primary monomial ideals. Therefore $\omega(I+J)\le \min\{\omega(I),\omega(J)\}\le\max\{\omega(I),\omega(J)\}\le \omega(I\cap J)\le\omega(IJ)$  by Lemma \ref{orderreversing}. On the other hand, let $I=\bigcap_{i=1}^{r}Q_{i}$ and $J=\bigcap_{j=1}^{s}T_{j}$ be standard decomposition of $I$ and $J$, respectively. Then $I\cap J=(\bigcap_{i=1}^{r}Q_{i})\cap(\bigcap_{j=1}^{s}T_{j})$ is an irreducible decomposition of $I\cap J$, and by throwing away any redundant component, there exists $A\subseteq \{1,\cdot\cdot\cdot, r\}$ and $B\subseteq \{1,\cdot\cdot\cdot, s\}$ so $I\cap J=(\bigcap_{i\in A}Q_{i})\cap(\bigcap_{j\in B}T_{j})$ is the standard decomposition of $I\cap J$. Thus by Corollary \ref{primary}, 
\begin{align*}
\omega(I\cap J)&=\max\{\max_{i\in A}\{e(Q_{i})\},\max_{j\in B}\{e(T_{j})\}\}\\
&\le \max\{\max_{1\le i\le r}\{e(Q_{i})\},\max_{1\le j\le s}\{e(T_{j})\}\}\\
&=\max\{\omega(I),\omega(J)\}.
\end{align*}
Moreover, $(\sqrt{IJ})^{e(I)+e(J)}=P^{e(I)+e(J)}=P^{e(I)}P^{e(J)}=(\sqrt{I})^{e(I)}(\sqrt{J})^{e(J)}\subseteq IJ$, and so  $e(IJ)\le e(I)+e(J)$. Combined with Lemma \ref{fund}, we have $\omega(IJ)\le \omega(I)+\omega(J)$.
It remains to show that $\omega(I:J)\ge \omega(I)-\omega(J)$. When $J\subseteq I$, then we have $I:J=R$ and $\omega(I:J)=0\ge \omega(I)-\omega(J)$ by Lemma \ref{orderreversing}. If $J\not\subseteq I$, then $I:J$ is $P$-primary by Corollary \ref{monprimary}, and since $J(I:J)\subseteq I$, we have $\omega(I:J)+\omega(J)\ge \omega(I)$ by the first part of this lemma, hence the claim.
\end{proof}

As Anderson and Badawi pointed out (\cite[Example 2.7]{Anderson}), the conclusion of Lemma \ref{idealop} does not hold in every ring $R$. We add, that even in a polynomial ring over a field, the conclusion of the above lemma may fail if we drop any part of the hypothesis.

\begin{ex} 
\label{tw}
Let $R=k[x,y,z]$ and $I=(x^{2}, xy, y^{2}, xz^{2})$ and $J=(x^{2}, xy, y^{2}, yz^{3})$, so that neither $I$ nor $J$ are primary ideals. The standard decompositions of $I,J$ and $I\cap J$ are
\begin{align*} 
I&=(x^{2}, y, z^{2})\cap (x, y^{2})\\
J&=(x, y^{2}, z^{3})\cap (x^{2}, y)\\
I\cap J&=(x, y^{2})\cap (x^{2}, y)\\
I+J&=(x,y)\cap (x, y^{2}, z^{3})\cap (x^{2}, y, z^{2}).
\end{align*}

Thus we have $\omega(I)=3$, $\omega(J)=4$,  $\omega(I\cap J)=2$ and $\omega(I+J)=4$, so that $\omega(I\cap J)<\omega(I+J)=\max\{\omega(I),\omega(J)\}$.
\end{ex}

\begin{ex} Let $R = k[x,y,z]$ and $I=(x,y)$ and $J=(y,z^{2})$, so that $I$ and $J$ are both primary, but $\sqrt{I} \neq \sqrt{J}$. Then we have $\omega(I)=1$, $\omega(J)=2$ and $\omega(I\cap J)=3$, so that $\omega(I\cap J)>\max\{\omega(I),\omega(J)\}$.
\end{ex}

\begin{coro}
\label{idealp}
Let $I$ be a primary monomial ideal of $R = k[x_{1},\ldots,x_{n}]$. Then for each $m \in \mathbb{N}$ we have $\omega(I^{m})\le m\omega(I)$.
\end{coro}

\begin{proof}
Follows immediately by induction on $m$ and Lemma \ref{idealop}.
\end{proof}

Next, we derive a characterization of primary monomial $\omega$-linear ideals.

\begin{lemm}
\label{limitratio}
Let $R=k[x_{1},\cdot\cdot\cdot, x_{n}]$ and $Q$ a primary monomial ideal of $R$, so that $G(Q)=\{x_{i_{1}}^{a_{1}},\cdot\cdot\cdot, x_{i_{r}}^{a_{r}}, f_{1},\cdot\cdot\cdot, f_{t}\}$ where  $a_{1},\ldots, a_{r} \in \mathbb{N}$, $1 \leq i_{1} < i_{2} < \cdots < i_{r} \leq n$ and $f_{1},\cdot\cdot\cdot, f_{t}$ are monomials in  $k[x_{i_{1}},\cdot\cdot\cdot, x_{i_{r}}]$. Choose $s \in \{1,\cdot\cdot\cdot, r\}$ so $a_{s}=\max_{1\le j\le r}\{a_{j}\}$.
\begin{enumerate}
\item
If $f_{1}=f_{2}=\cdot\cdot\cdot=f_{t}=0$
, then $\omega(Q^{m})=(m-1)a_{s}+\omega(Q)$ for each $m \in\mathbb{N}$.
\item
$Q$ is $\omega$-linear if and only if $\omega(Q)=a_{s}$.
\end{enumerate}
\end{lemm}

\begin{proof}
\begin{enumerate}
\item Let $Q=(x_{i_{1}}^{a_{1}},\cdot\cdot\cdot, x_{i_{r}}^{a_{r}})$. Then given $m \in \mathbb{N}$, set $S_{m}=\{(k_{1},\cdot\cdot\cdot,k_{r})\in \mathbb{N}^{r}\mid \sum\limits_{j=1}^{r}k_{j}=m+r-1\}$ and $Q_{k}=(x_{i_{1}}^{k_{1}a_{1}},\cdot\cdot\cdot, x_{i_{r}}^{k_{r}a_{r}})$ for each $k=(k_{1},\cdot\cdot\cdot,k_{r}) \in S_{m}$. Then $Q^{m}=\cap_{k \in S_{m}}Q_{k}$ (\cite[Theorem 6.2.4]{Moore}). 
Now by Corollary \ref{primary} and Lemma \ref{powermonomial}, $\omega(Q^{m})=\max_{k \in S_{m}}\{e(Q_{k})\}=\max_{k\in S_{m}}\{\sum\limits_{j=1}^{r} k_{j}a_{j}\}-r+1=(m-1)a_{s}+\omega(Q)$.
\item
Fix $m \in \mathbb{N}$ and 
set $I_{1}=(x_{i_{1}}^{a_{1}},\cdot\cdot\cdot, x_{i_{r}}^{a_{r}})^{m}, I_{2}=(x_{i_{1}},\cdot\cdot\cdot, x_{i_{s-1}},x_{i_{s}}^{ma_{s}}, x_{i_{s+1}}, \cdot\cdot\cdot, x_{i_{r}})$. It follows that $I_{1}\subseteq Q^{m}\subseteq I_{2}$ are $(x_{i_{1}},\cdot\cdot\cdot, x_{i_{r}})$-primary ideals, so we have\\
$ma_{s}=\omega(I_{2})\le\omega(Q^{m})\le\omega(I_{1})=(m-1)a_{s}+\sum_{j=1}^{r}a_{j}-r+1$ by Corollary \ref{orderreversing}, Lemma \ref{powermonomial} and part 1 of this lemma. Therefore if $Q$ is $\omega$-linear, then $\omega(Q)=\displaystyle \lim_{m\to\infty}\cfrac{m\omega(Q)}{m}= \displaystyle \lim_{m\to\infty}\cfrac{\omega(Q^{m})}{m}=a_{s}$. Conversely, suppose that $\omega(Q)=a_{s}$ and fix $m\in\mathbb{N}$. Then since $x_{i_{s}}^{ma_{s}}\in G(Q^{m})$ we have $\omega(Q^{m})\ge ma_{s}=m\omega(Q)$ by Lemma \ref{topi}. Hence $\omega(Q^{m})=m\omega(Q)$ by Corollary \ref{idealp} and so $Q$ is $\omega$-linear.
\end{enumerate}
\end{proof}



\begin{coro}
\label{return}
Let $I$ be an irreducible monomial ideal of $R=k[x_{1},\cdot\cdot\cdot, x_{n}]$ so that  $I=(x_{i_{1}}^{a_{1}},\cdot\cdot\cdot, x_{i_{r}}^{a_{r}})$ for some $a_{1},\ldots, a_{n} \in \mathbb{N}$. Set $a_{s}=\max_{1\le j\le r}\{a_{j}\}$. Then the following are equivalent.
\begin{enumerate}
\item
$I$ is $\omega$-linear.
\item
$\omega(I^{m})=m\omega(I)$ for some $m>1$.
\item
$\omega(I)=a_{s}$.
\item
$a_{i}=1$ for each $i\neq s$.
\end{enumerate}
\end{coro}

\begin{proof}
$1\Rightarrow 2$: Obvious.\\
$2\Rightarrow 3$: Suppose that $\omega(I^{m})=m\omega(I)$ for some $m >1$. By Lemma \ref{limitratio}.1 we have $\omega(I^{m})=(m-1)a_{s}+\omega(I)$. Hence $\omega(I)=a_{s}$.\\
$3\Leftrightarrow 4$: Immediate consequence of Lemma \ref{powermonomial}.\\
$3\Leftrightarrow 1$: Follows from Lemma \ref{limitratio}.2.
\end{proof}

\begin{lemm}
Let $P$ be a monomial prime ideal of $R$. If $I,J$ are $P$-primary $\omega$-linear monomial ideals of $R$, then so is $I\cap J$.
\end{lemm}

\begin{proof}
Without loss of generality we may assume that $\omega(I)\ge \omega(J)$.   
By Lemma \ref{limitratio}.2,  there is $j\in\{1,\cdot\cdot\cdot, r\}$ so that $x_{i_{j}}^{\omega(I)}\in G(I)$. There exists $a\in\mathbb{N}$ so $x_{i_{j}}^{a}\in G(J)$. Then again, by Lemma \ref{limitratio}.2, $a\le \omega(J)$. Now, $x_{i_{j}}^{\omega(I)}=lcm(x_{i_{j}}^{\omega(I)}, x_{i_{j}}^{a}) \in G(I\cap J)$. On the other hand, $\omega(I\cap J)=\omega(I)$ by Lemma \ref{idealop}. Hence $I\cap J$ is $\omega$-linear by Lemma \ref{limitratio}.2.
\end{proof}

Given a monomial ideal $I$ of $R=k[x,y]$ we will write $I=(x^{a_{1}}y^{b_{1}}, \ldots, x^{a_{r}}y^{b_{r}})$ where $\{a_{i}\}$ and $\{b_{i}\}$ are strictly decreasing and strictly increasing sequences of non-negative integers, respectively. Similarly, if $J$ is a monomial ideal of $R$ we write $J=(x^{c_{1}}y^{d_{1}}, \ldots, x^{c_{s}}y^{d_{s}})$ where $\{c_{i}\}$ and $\{d_{i}\}$ are strictly decreasing and strictly increasing sequence of non-negative integers, respectively. Hence $b_{1}=a_{r}=0$ iff $I$ is $(x,y)$-primary, and $d_{1}=c_{s}=0$ iff $J$ is $(x,y)$-primary.

\begin{lemm}
\label{2dimex}
Let $R=k[x,y]$ and $I,J$ be $(x,y)$-primary monomial ideals with $\omega(I)\ge \omega(J)$. Then $\omega(IJ)\le \omega(I)+\max\{c_{1}, d_{s}\}$.
\end{lemm}

\begin{proof}

We may assume that $c_{1}\ge d_{s}$. Then $e(I)=\omega(I)\ge\omega(J)\ge c_{1}$ by Lemma \ref{topi}, so $(x,y)^{e(I)+c_{1}}=(x,y)^{e(I)}(x^{c_{1}}, y^{c_{1}})=(\sqrt{I})^{e(I)}(x^{c_{1}}, y^{c_{1}})\subseteq IJ$ are $(x,y)$-primary ideals. Therefore $\omega(IJ)\le \omega((x,y)^{e(I)+c_{1}})=e(I)+c_{1}=\omega(I)+c_{1}$ by Lemma \ref{orderreversing}.
\end{proof}

We now classify $\omega$-linear monomial ideals $I$ in $R=k[x,y]$.

\begin{lemm}
\label{twodim}
Let  $R=k[x,y]$ and $I=(x^{a_{1}}y^{b_{1}},\cdot\cdot\cdot, x^{a_{r}}y^{b_{r}})$ be a monomial ideal of $R$. Then the following are equivalent.
\begin{enumerate}
\item
$I$ is $\omega$-linear.
\item
$\omega(I^{n})=n\omega(I)$ for some $n>1$.
\item
$\omega(I)=\max\{a_{1}+b_{1}, a_{r}+b_{r}\}$.
\end{enumerate}
\end{lemm}

\begin{proof}
Note that given $n\in \mathbb{N}$ and a monomial $f$ of $R$,  by Lemma \ref{irreducible} we have 
\begin{align*}
\omega(I^{n})&=n\omega(I)\\
 \Leftrightarrow n(deg(f))+\omega(I^{n})&=n(deg(f))+n\omega(I)\\
 \Leftrightarrow deg(f^{n})+\omega(I^{n})&=n(deg(f)+\omega(I))\\  \Leftrightarrow \omega((fI)^{n})&=n\omega(fI).
\end{align*} 

Moreover, if $I$ is a principal ideal, then $I$ satisfies all of 1,2 and 3 by Corollary \ref{prinmo}. 
Hence we may assume that $I$ is a $(x,y)$-primary monomial ideal of $R$. That is, $a_{r}=b_{1}=0$.\\
\\
$1\Rightarrow 2$ is trivial.\\
$2\Rightarrow 3$: Suppose that $\omega(I^{n})=n\omega(I)$ for some $n>1$. Note that $\omega(I^{n-1})+\omega(I)\ge \omega(I^{n})=n\omega(I)$ by Lemma \ref{idealop} and $\omega(I^{n-1})\le(n-1)\omega(I)$ by Corollary \ref{idealp}, and thereby $\omega(I^{n-1})=(n-1)\omega(I)$. Hence we must have $\omega(I^{2})=2\omega(I)$. 
Since $\omega(I^{2})\le \omega(I)+\max\{a_{1}, b_{r}\}$ by Lemma \ref{2dimex}, 
$\omega(I)=\omega(I^{2})-\omega(I)\le \max\{a_{1}, b_{r}\}$. On the other hand, $\omega(I)\ge\max\{a_{1}, b_{r}\}$ by Lemma \ref{topi}. Therefore $\omega(I)=\max\{a_{1}, b_{r}\}$.\\
$3 \Rightarrow 1$: Follows from Lemma \ref{limitratio}.2.
\end{proof}

\begin{lemm}
\label{intomega}
The set of monomial $\omega$-linear ideals of $R=k[x,y]$ is multiplicatively closed.
\end{lemm}

\begin{proof}
Let $I$ and $J$ be monomial $\omega$-linear ideals of $R$. By Lemma \ref{irreducible} we may assume that $I$ and $J$ are $(x,y)$-primary ideals of $R$. 
Then $\omega(I)= \max\{a_{1}, b_{r}\}$, $\omega(J)=\max\{c_{1}, d_{s}\}$ by Lemma \ref{twodim}. Now, $x^{a_{1}+c_{1}}$ and $y^{b_{r}+d_{s}}$ are elements of $G(IJ)$. Hence by Lemma \ref{limitratio}.2 and Lemma \ref{topi}, to show that $IJ$ is $\omega$-linear it suffices to show that $\omega(IJ)\le\max\{a_{1}+c_{1}, b_{r}+d_{s}\}$. Suppose that $\omega(I)=a_{1}$ and $\omega(J)=c_{1}$. Then 
all we have to show is $\omega(IJ)\le a_{1}+c_{1}$, which follows from Lemma \ref{idealop}. The case when $\omega(I)=b_{r}$ and $\omega(J)=d_{s}$ can be derived in the exact same manner. Therefore, without loss of generality we may assume that $\omega(I)=a_{1}>b_{r}$ and $\omega(J)=d_{s}>c_{1}$. Observe now that  $Ix^{c_{1}} + Jy^{b_{r}}$ 
is an $(x,y)$-primary ideal contained in $IJ$. Thus  by Lemma \ref{orderreversing} and Theorem \ref{hoo} we have
\begin{align*}
\omega(IJ)&\le \omega(Ix^{c_{1}}+Jy^{b_{r}})\\
&=\max\{\max_{1\le i\le r-1}\{a_{i}+b_{i+1}+c_{1}\}-1, \max_{1\le j\le s-1}\{c_{j}+d_{j+1}+b_{r}\}-1\}\\
&=\max\{ \omega(I)+c_{1}, \omega(J)+b_{r}\}\\
&=\max\{a_{1}+c_{1}, b_{r}+d_{s}\}.\end{align*}\end{proof}

Recall that given an ideal $I$ of a commutative ring $R$, an element $f \in R$ is said to be \textit{integral} over $I$ if there is some $k \in \mathbb{N}$ and $c_{i} \in I^{i}$ for each $i \in \{1, \ldots, k\}$ so that
\begin{align*} 
f^{k}+c_{1}f^{k-1}+\cdots + c_{k-1}f+c_{k}&=0.
\end{align*}
The set of elements of $R$ integral over $I$ is called the \textit{integral closure} of $I$ and denoted by $\overline{I}$. $I$ is said to be \textit{integrally closed} if $I=\overline{I}$. 

\begin{coro}
\label{int}
Every integrally closed monomial ideal of $R=k[x,y]$ is $\omega$-linear.
\end{coro}

\begin{proof}
Let $I$ be an integrally closed monomial ideal of $R$. It is well known that $R$ is an \textit{integrally closed domain} (i.e., $R$ is an integral domain  that contains every nonzero element of the quotient field of $R$ that is integral over $R$), and that each principal ideal of $R$ is integrally closed, and the product of an integrally closed ideal of $R$ and a nonzero element of $R$ yields another integrally closed ideal of $R$. 
 Hence by Lemma \ref{irreducible} we may assume that $I$ is $(x,y)$-primary. Now by \cite[Proposition 2.6]{Quinonez} there are monomial ideals $I_{1}=(\{x^{r-i}y^{b_{i}}\}_{i=0}^{r})$ and $I_{2}=(\{x^{a_{i}}y^{i}\}_{i=0}^{r})$ of $R$ with $0=b_{0}<b_{1}<\cdot\cdot\cdot< b_{r}$ and $a_{0}>a_{1}>\cdot\cdot\cdot> a_{r}=0$ so $I=I_{1}I_{2}$. Thus
 by Lemma \ref{intomega}, it suffices to show that $I_{1}$ and $I_{2}$ are $\omega$-linear. By Theorem \ref{hoo}, $\omega(I_{1})=\max_{0\le i\le r-1}\{c_{i}\}$, where $c_{i}=r-i+b_{i+1}-1$ for each $i\in\{0,1,\cdot\cdot\cdot, r-1\}$. Since $c_{i+1}-c_{i}=b_{i+1}-(b_{i}+1)\ge0$ for each $i\in\{0,1,\cdot\cdot\cdot, r-1\}$, we have $\omega(I_{1})=c_{r-1}=b_{r}=\max\{r,b_{r}\}$ and $I_{1}$ is $\omega$-linear by Lemma \ref{twodim}. The proof that $I_{2}$ is $\omega$-linear follows similarly.
\end{proof}

\begin{rem}
\label{rem3}
\begin{enumerate}
\item
Even if $I$ and $J$ are $\omega$-linear monomial primary ideals such that $\sqrt{I}=\sqrt{J}$, we may have $\omega(I\cap J)<\omega(IJ)<\omega(I)+\omega(J)$. Indeed, set $R=k[x,y]$, $I=(x^{3}, xy, y^{2})$ and $J=(x^{2}, xy, y^{3})$. Then both $I$ and $J$ are $\omega$-linear $(x,y)$-primary ideals of $R$. However, $IJ=(x^{5}, x^{3}y, x^{2}y^{2}, xy^{3}, y^{5})$, so  $\omega(IJ)=5<6=\omega(I)+\omega(J)$. On the other hand, $\omega(I\cap J)=\max\{\omega(I),\omega(J)\}=3$ by Corollary \ref{primary}. 
\item
Not every $\omega$-linear monomial ideal of $R=k[x,y]$ is integrally closed. For example, set $I=(x^{3}, xy^{2}, y^{4})$. Then $\omega(I)=4$ by Theorem \ref{hoo}, and $I$ is $\omega$-linear by Lemma \ref{twodim}. However, $(x^{2}y)^{2}=x^{3}(xy^{2})\in I^{2}$ and $x^{2}y\not\in I$. Thus $I$ is not integrally closed (\cite[Theorem 1.4.2]{Herzog}). 
\end{enumerate}
\end{rem}

So far, we have considered only primary $\omega$-linear monomial ideals, and most of the proof is solely based on the fact that $e(I)=\omega(I)$ when $I$ is a primary ideal. We now show that there exists a class of (integrally closed) nonprimary $\omega$-linear monomial ideals. In fact, some of the squarefree monomial ideals are $\omega$-linear, as we will see in the next two lemmas.

Recall that a graph $G$ consists of a set of vertices $V = \{v_{1}, . . . , v_{n}\}$ and a set of edges $E \subseteq \{v_{i}v_{j} | v_{i}, v_{j} \in V \}$, and is called \textit{bipartite} if there exists two disjoint subsets $U_{1}, U_{2}$ of $V$ such that $E\subseteq\{v_{i}v_{j}\mid v_{i}\in U_{1}, v_{j}\in U_{2}\}$.
The \textit{edge ideal} of $G$ is defined to be the ideal $I = (\{x_{i}x_{j} | v_{i}v_{j} \in E\})$ of $R = k[x_{1}, . . . , x_{d}]$, where $k$ be a field and $d$ is the number of vertices of $G$. Given a graph $G=(V,E)$, a subset $W$ of $V$ is said to be a \textit{vertex cover} if given $v_{i}v_{j}\in E$, either $v_{i}\in W$ or $v_{j}\in W$. A vertex cover $W$ of $G$ is said to be a \textit{minimal vertex cover} if each proper subset of $W$ is not a vertex cover of $G$. 

If $I$ is an edge ideal of a graph, then it is a squarefree monomial ideal and 
a monomial prime ideal $P$ is a minimal ideal of $I$ if and only if the set of vertices that corresponds to $P$ is a minimal vertex cover. Also, a graph is bipartite if and only if it has no cycle of odd length as its subgraph. 

Our first example of a nonprimary $\omega$-linear ideal is the edge ideal of a  bipartite graph.
\begin{lemm}
\label{bipartite}
Let $R=k[x_{1},\cdot\cdot\cdot, x_{n}]$. If $I$ is an ideal of $R$ that is also the edge ideal of a bipartite graph $G$, then $I$ is $\omega$-linear.
\end{lemm}

\begin{proof}
Let $I$ be an edge ideal of a graph $G$ and let $P_{1}, \cdot\cdot\cdot, P_{r}$ be the set of (incomparable) minimal prime ideals of $I$.
Recall that a graph $G$ is bipartite if and only if \[I^{m}=\bigcap_{P \text{ is a minimal prime of $I$}}P^{m}\] for each $m\in\mathbb{N}$ (\cite[Theorem 5.9]{Simis}). 
Hence if $G$ is bipartite, then by Lemma \ref{upper}, $\omega(I^{m})=\sum\limits_{i=1}^{r} e(P_{i}^{m})=\sum\limits_{i=1}^{r}m=mr$ for each $m \in\mathbb{N}$. Therefore the conclusion follows. 
\end{proof}

There are nonbipartite graphs whose edge ideals are $\omega$-linear.

\begin{lemm}
\label{hum}
Let $R=k[x_{1},\cdot\cdot\cdot, x_{n}]$. Let  $I=(x_{1}x_{2}, x_{2}x_{3},\cdot\cdot\cdot, x_{n-1}x_{n}, x_{n}x_{1})$ (that is, $I$ is  the edge ideal of a cycle graph of length $n$). Then $I$ is $\omega$-linear.
\end{lemm}

\begin{proof}
Since a cycle of even length is bipartite, by Lemma \ref{bipartite} we may assume that  $n=2l+1$ for some $l\in\mathbb{N}$. Fix $m\in \mathbb{N}$. $I$ is a squarefree monomial ideal, so $I=P_{1}\cap\cdot\cdot\cdot \cap P_{r}$ where $P_{1},\cdot\cdot\cdot P_{r}$ are the minimal prime ideals of $I$ (\cite[Lemma 1.3.5]{Herzog}). Thus by Lemma \ref{upper} we have $\omega(I)=\sum_{i=1}^{r}e(P_{i})=r$, and we only need to show that $\omega(I^{m})=mr$. 
Note that since $I$ is an edge ideal of a cycle of length $2l+1$, Ass$(R/I^{m})=\{P_{1},\cdot\cdot\cdot, P_{r}\}$ if $m\le l$ and Ass$(R/I^{m})=\{P_{1},\cdot\cdot\cdot, P_{r}, \mathfrak{m}\}$ if $m>l$ (\cite[Lemma 3.1]{Chen}). Hence if $m\le l$, then $I^{m}=\cap_{i=1}^{r}P_{i}^{m}$ and $\omega(I^{m})=\sum\limits_{i=1}^{r}e(P_{i}^{m})=mr$ by Lemma \ref{upper}, so we are done. Assume that $m>l$. Then $I^{m}=(\cap_{i=1}^{r}P_{i}^{m})\cap Q$ is the canonical primary decomposition of $I^{m}$, where $Q$ is an $\mathfrak{m}$-primary monomial ideal of $R$ (\cite[Proposition 1.4.4]{Herzog}). 
Now, $Q=(x_{1}^{a_{1}},\cdot\cdot\cdot, x_{n}^{a_{n}}, f_{1},\cdot\cdot\cdot, f_{t})$ for some $a_{i} \in \mathbb{N}$ and monomials $f_{i}$. Since $I$ is a squarefree monomial ideal and $Q$ is a primary component of $I^{m}$, we must have $a_{i}\le m$ for each $i\in\{1,\cdot\cdot\cdot, n\}$ 
and then $e(Q)\le e((x_{1}^{a_{1}},\cdot\cdot\cdot, x_{n}^{a_{n}}))\le mn-n+1\le mr$ by 
Lemma \ref{powermonomial} and since $n \leq r$. It follows that $\omega(I^{m})=\max\{\sum\limits_{i=1}^{r}e(P_{i}^{m}), e(Q)\}=\max\{mr, e(Q)\}=mr$ by Lemma \ref{l} and Lemma \ref{upper}.
\end{proof}

We were unable to show whether every edge ideal is $\omega$-linear, but we do have the following.

\begin{lemm}
Let $I$ be a square-free monomial ideal. Then $\omega(I^{m})\ge m\omega(I)$ for each $m\in\mathbb{N}$.
\end{lemm}

\begin{proof}
Let $P_{1},\cdot\cdot\cdot, P_{r}$ be minimal prime ideals of $I$. Then $I=\cap_{i=1}^{r}P_{i}$  and $\omega(I)=r$ by Lemma \ref{upper}. 
Set $f_{i}=\sum\limits_{x_{j}\in G(P_{i})}x_{j}$ for each $i\in\{1,\cdot\cdot\cdot, r\}$. Then $f:=\prod_{i=1}^{r}f_{i}\in \prod_{i=1}^{r}P_{i}\subseteq I$, so $f^{m}\in I^{m}$. However, $\cfrac{f^{m}}{f_{i}}\not\in P_{i}^{m}$, so $\cfrac{f^{m}}{f_{i}}\not\in I^{m}$(\cite[Proposition 1.4.4]{Herzog}). Thus $I^{m}$ is not $(mr-1)$-absorbing and $\omega(I^{m})\ge m\omega(I)$.
\end{proof}

We close the section with 
the following question: Is every integrally closed monomial ideal $\omega$-linear?
Integrally closed monomial ideals considered in this note (monomial ideals in $R=k[x,y]$, irreducible monomial ideals, or edge ideal of bipartite graphs) were all $\omega$-linear. Note also that if this question has an affirmative answer, then it follows that every edge ideal is $\omega$-linear.

\end{document}